\documentclass[11pt, leqno]{amsart}
\usepackage{amsmath}
\usepackage{amsfonts}
\usepackage{amssymb}
\usepackage{dsfont}
\usepackage{graphicx}
\usepackage{mathrsfs}
\usepackage[margin=2.90cm]{geometry}
\usepackage{url}
\usepackage{hyperref}
\providecommand{\U}[1]{\protect\rule{.1in}{.1in}}
\newtheorem{theorem}{Theorem}[section]
\newtheorem*{acknowledgement*}{Acknowledgements}

\newtheorem{lemma}[theorem]{Lemma}

\newtheorem{proposition}[theorem]{Proposition}
\newtheorem{remark}[theorem]{Remark}

\def\<{\left\langle}
\def\>{\right\rangle}

\newcommand{\matrice}{\begin{pmatrix}}
\newcommand{\ok}{\end{pmatrix}}
\newcommand{\dmatrice}{\begin{vmatrix}}
\newcommand{\dok}{\end{vmatrix}}

\def\<{\left\langle}
\def\>{\right\rangle}

\begin{document}
\title[Poincar\'e inequality and topological rigidity]{Poincar\'e inequality and topological rigidity of translators and self-expanders for the mean curvature flow}

\thanks{D. Impera is partially supported by INdAM-GNSAGA and M. Rimoldi is partially supported by INdAM-GNAMPA. The authors acknowledge partial support by the PRIN 2022 project ``Real and Complex Manifolds: Geometry and Holomorphic Dynamics'' - 2022AP8HZ9}



\author[Debora Impera]{Debora Impera}
\address[Debora Impera]{Dipartimento di Scienze Matematiche ``Giuseppe Luigi Lagrange", Politecnico di Torino, Corso Duca degli Abruzzi, 24, Torino, Italy, I-10129}
\email{debora.impera@polito.it}

\author[Michele Rimoldi]{Michele Rimoldi}
\address[Michele Rimoldi]{Dipartimento di Scienze Matematiche ``Giuseppe Luigi Lagrange", Politecnico di Torino, Corso Duca degli Abruzzi, 24, Torino, Italy, I-10129}
\email{michele.rimoldi@polito.it}


\begin{abstract}
We prove an abstract structure theorem for weighted manifolds supporting a weighted $f$-Poincar\'e inequality and whose ends satisfy a suitable non-integrability condition. We then study how our arguments can be used to obtain full topological control on two important classes of hypersurfaces of the Euclidean space, namely translators and self-expanders for the mean curvature flow, under either stability or curvature asumptions. 
As an important intermediate step in order to get our results we get the validity of a Poincar\'e inequality with respect to the natural weighted measure on any translator and we prove that any end of a translator must have infinite weighted volume. Similar tools can be obtained for properly immersed self-expanders permitting to get topological rigidity under curvature assumptions.
\end{abstract}

\maketitle

\section{Introduction and main results}\label{Sect_Intro}

\subsection{Main definitions} Translators for the mean curvature flow give rise to a special class of eternal solutions for the flow that, besides having their own intrinsic interest, are models of Type II singularities when starting from an initial mean convex closed hypersurface. An isometrically immersed complete (orientable) hypersurface of the Euclidean space $x:\Sigma^m\to(\mathbb{R}^{m+1}, \<\,,\,\>)$ is said to be a translator for the mean curvature flow if its mean curvature vector field satisfies the equation
\begin{equation}\label{T}
\mathbf{H}=\bar{V}^{\bot}
\end{equation}
for some parallel unit length vector field $\bar{V}$ in $\mathbb{R}^{m+1}$, where $(\cdot)^{\bot}$ denotes the projection on the normal bundle of $\Sigma$.  Letting $\left\{\bar{E}_{1},\ldots,\bar{E}_{m+1}\right\}$ be the standard orthonormal basis of $\mathbb{R}^{m+1}$, we assume from now on, without loss of generality, that $\bar{V}=\bar{E}_{m+1}$ in \eqref{T}. 

One can actually look at these objects also from a variational point of view. Indeed, it is by now well-known that equation \eqref{T} turns out to be the Euler-Lagrange equation for the weighted volume functional 
\[
\ \mathrm{vol}_{f}(\Sigma)=\int_{\Sigma}e^{-f}d\mathrm{vol}_{\Sigma},
\]
when choosing $f=-\<x,\bar{E}_{m+1}\>=-x_{m+1}$. This variational characterization permits to talk about stability properties and these latters are understood by looking at the second variation for the weighted area functional. In particular, since the Bakry-\`Emery Ricci curvature $\overline{\mathrm{Ric}}_{f}\doteq \overline{\mathrm{Ric}}+\overline{\mathrm{Hess}}(f)$ of the ambient Euclidean space identically vanishes, stability properties are taken into account by spectral properties of the weighted Jacobi operator of $\Sigma$ given by
\[
\ L_{f}\doteq -\Delta_{f}-\|A\|^{2},
\]
where $\Delta_{f}\doteq \Delta-\langle\nabla f, \nabla\rangle$, $\Delta\doteq\mathrm{div}(\nabla)$; for more details see e.g. \cite{IR_fMin}. It can be easily proved, by an adaptation of a result by Fischer-Colbrie and Schoen, \cite{FCS}, that translators with mean curvature that does not change sign (i.e. those which are graphical w.r.t. the translating direction $\bar{E}_{m+1}$) are either $f$-stable or have identically vanishing mean curvature (see \cite{ImperaRimoldi_Transl}). In this latter case they actually split as the product of a line parallel to the translating direction and a stable minimal hypersurface in the orthogonal complement of this line. More generally, a similar reasoning actually yields that translators which are graphical w.r.t. to any direction are $f$-stable. 

In recent years a great effort has been made on one hand to construct a wide panorama of examples and on the other hand to obtain classification results for these objects. Indeed plenty of examples of translators were constructed by various authors, and these go far beyond the graphical setting and the trivial topological structure; see e.g. \cite{CSS}, \cite{DDPN}, \cite{N1}, \cite{HMW22}, \cite{HMW_Annuli}, \cite{HMW_Scherk} and e.g. \cite{GMM22} for an almost updated survey on known examples in the $2$-dimensional case. Moreover, exploiting fundamental results in \cite{W} and \cite{SX}, it was recently reached the full classification of 2-dimensional vertical graphical translators, \cite{HIMW}, and of $2$-dimensional semigraphical translators, \cite{HMW22}. 

\medskip

Another fundamental class of hypersurfaces to be understood is that of self-expanders for mean curvature flow. These objects are expected to describe both the asymptotic long term behaviour of the flow and the local structure after singularities in the very short time, \cite{EckerHuisken}, \cite{Stavrou} \cite{AIC}. An isometrically immersed complete (orientable) hypersurface of the Euclidean space $x:\Sigma^m\to(\mathbb{R}^{m+1}, \<\,,\,\>)$ is said to be a self-expander for the mean curvature flow if its mean curvature vector field satisfies the equation
\begin{equation}\label{SE}
\mathbf{H}=x^{\bot}.
\end{equation}
Lately there was much activity in literature for studying geometric properties of self-expanders, with a particular focus on those which are asimptotically conical. For some key references and recent works on self-expanders see e.g. \cite{I}, \cite{Ding} \cite{BW1}, \cite{BW2}, \cite{Helmensdorfer}, \cite{Smoczyk}, \cite{FMG}, \cite{ChengZhou_Selfexp}, \cite{AncariCheng_GeoDed}.

It is well-known that also self-expanders have a variational interpretation, being critical points of the weighted volume functional $\mathrm{vol}_{f}$ when choosing $f=-\frac{|x|^2}{2}$. Some results  about the spectrum of the corresponding weighted Laplacian $\Delta_f$ where proved in \cite{ChengZhou_Selfexp}. In particular, in \cite{ChengZhou_Selfexp} it is obtained the discreteness of the spectrum for properly immersed self-expanders and a sharp universal lower bound for the bottom of the spectrum. 
\medskip

The present paper is devoted to the study of topological rigidity properties of translators and self-expanders for the mean curvature flow in the Euclidean space under either stability or curvature assumptions. From an abstract viewpoint the key ideas in order to get our results, are contained in the proof of an abstract theorem of independent interest which we state in the general setting of weighted manifolds. 

\subsection {An abstract structure theorem for weighted manifolds supporting a weighted Poincar\'e inequality} Recall that, letting $(M, g)$ be a complete Riemannian manifold and given a smooth function $f$ on $M$, we can consider the weighted manifold $M_{f}=(M, g, e^{-f}d\mathrm{vol})$. Associated to the weighted manifold $M_f$ there is a natural divergence form second order diffusion operator, the $f$-Laplacian, defined on $u$ as
\[
\ \Delta_f u=e^{f}\mathrm{div}(e^{-f}\nabla u)=\Delta u- g(\nabla f, \nabla u).
\]
This is clearly symmetric on the space $L^{2}(M_f)$ of square integrable function w.r.t. the weighted measure. The geometry of weighted manifolds is visible in the weighted metric structure, i.e. in the weighted measure of (intrinsic) metric objects, and it is controlled by suitable concepts of curvature adapted to the density of the measure. An important generalization of the Ricci tensor in this setting is the so-called Bakry-\'Emery Ricci tensor, defined as
\[
\ \mathrm{Ric}_{f}=\mathrm{Ric}+\mathrm{Hess}(f).
\]
\medskip

Structure theorems for complete Riemannian manifolds satisfying both a Ricci curvature lower bound and a (weigthed) Poincar\'e inequality were classically investigated starting form papers by Li and Wang; \cite{LiWang_1} \cite{LiWang_Poincare}. Usually, these results permit to get geometric information about connected components at infinity (i.e. ends). Moreover, note that in presence of a genuine Poincar\'e inequality, if all ends have infinite volume, a key role in controlling the topology at infinity is played by the first de Rham's cohomology group with compact support; see \cite[Proposition 2.4]{Carron_LectNotes}. Here we combine these ideas, in the setting of weighted manifolds, getting that when a weighted manifold $M_f$ supports a weighted $f$-Poincar\'e inequality and its ends satisfy a certain weighted non-integrability condition, then a suitable lower bound on the Bakry-\'Emery Ricci curvature implies topological rigidity. This is the content of the following

\begin{theorem}\label{MainA}
Let $(M,g,e^{-f}d\mathrm{vol})$ be a complete weighted manifold satisfying the weighted $f$-Poincar\'e inequality
\begin{equation}\label{eq_weightedpoinc}
\int_M\rho u^2e^{-f}d\mathrm{vol}\leq \int_M |\nabla u|^2e^{-f}d\mathrm{vol},\qquad \forall u\in\mathrm{C}^\infty_c(M),
\end{equation}
where $\rho$ is a positive smooth function on $M$ satisfying, on any end $E$ of $M$,
\begin{equation}\label{Nonintrho}
\ \int_{E}\rho e^{-f}d\mathrm{vol}=+\infty.
\end{equation}
Assume that
\[\mathrm{Ric}_f\geq-a(x),\]
for some smooth function $a$ satisfying $a(x)\leq\frac{\rho(x)}{\delta}$, for some $\delta>1$. Then:
\begin{enumerate}
\item  $M$ is connected at infinity (that is, $M$ has only one end);
\item $M$ admits no codimension one cycle that does not disconnect $M$. In particular, if $\mathrm{dim}(M)=2$ then $M$ has genus zero.
\end{enumerate}
\end{theorem}

\begin{remark}
\rm{Besides the fact that we are working in the setting of weighted manifolds, there are two main novelties in the last statement we would like to stress:
\begin{itemize}
\item[(i)] Differently from \cite{Carron_LectNotes} we are assuming a truly weighted Poincar\`e inequality \eqref{eq_weightedpoinc} with non-necessarily constant $\rho$. Assumption \eqref{Nonintrho} is the suitable replacement in this setting for the request that all ends have infinite volume. 
\item[(ii)] Passing again through the first de Rham's cohomology group with compact support, thanks to point (2) of Theorem \ref{MainA}, we can actually get a full control on the topology of the weighted manifold. The proof of this fact exploits again ideas of Carron; see \cite{Carron_LectNotes} and \cite[Proposition 3.3]{KS} for a precise statement.
\end{itemize}}
\end{remark}

\subsection{Topological rigidity for translators and self-expanders} In the second part of the paper we apply the ideas of the proof of Theorem \ref{MainA} to obtain analogous conclusions for translators and self-expanders for the mean curvature flow. These two classes of hypersurfaces turn out to fit very well the theory. Indeed, we can prove the following two results.
\begin{theorem}\label{Thmfpoincaretranslato}
Let $x:\Sigma^{m\geq 2}\to\mathbb{R}^{m+1}$ be a translator for the mean curvature flow and let $f=-x_{m+1}$. Then:
\begin{itemize}
\item[(a)] the following $f$-Poincar\'e inequality holds
\begin{equation*}
\frac{1}{4}\int_\Sigma \varphi^2 e^{-f}d\mathrm{vol}_{\Sigma}\leq \frac{1}{4}\int_\Sigma (1+H^2) \varphi^2 e^{-f}d\mathrm{vol}_{\Sigma}\leq \int_\Sigma |\nabla \varphi|^2 e^{-f}d\mathrm{vol}_{\Sigma},
\end{equation*}
for any $\varphi\in C^{\infty}_c(\Sigma)$;
\item[(b)] every end of $\Sigma$ has infinite $f$-volume.
\end{itemize}
\end{theorem}

\begin{theorem}\label{ThmExpTools}
Let $x:\Sigma^{m\geq 2}\to\mathbb{R}^{m+1}$ be a self-expander of the mean curvature flow and let $f=-|x|^2/2$. Then:
\begin{itemize}
\item[(a)] the following $f$-Poincar\'e inequality holds
\begin{equation*}
m\int_\Sigma \varphi^2 e^{-f}d\mathrm{vol}_{\Sigma}\leq \int_\Sigma (m+H^2) \varphi^2 e^{-f}d\mathrm{vol}_{\Sigma}\leq \int_\Sigma |\nabla \varphi|^2 e^{-f}d\mathrm{vol}_{\Sigma},
\end{equation*}
for any $\varphi\in C^{\infty}_c(\Sigma)$;
\item[(b)] if $x$ is a proper immersion then every end of $\Sigma$ has infinite $f$-volume.
\end{itemize}
\end{theorem}

\begin{remark}
\rm{Note that point (a) in Theorem \ref{ThmExpTools} was originally proved in \cite{ChengZhou_Selfexp}, whereas Theorem \ref{Thmfpoincaretranslato} and also the point (b) in Theorem \ref{ThmExpTools} are, to the best of our knowledge,  completely new.}
\end{remark}

In the setting of translators, by Gauss' equation, a direct computation yields that
\[
\ \mathrm{Ric}_{f}^{\Sigma}\geq -\|A\|^2,
\]
where we are denoting by $A$ the second fundamental form of the immersion. Hence, by Theorem \ref{Thmfpoincaretranslato}, one could apply the abstract Theorem \ref{MainA} to prove topological rigidity under the curvature assumption
\[
\ \|A\|^2\leq \frac{1+H^2}{4\delta},
\]
with $\delta>1$. On the other hand, exploiting rigidity results for translators, it is possible to weaken the curvature assumption to the limit case $\delta=1$. In this respect note that, assuming $\delta=1$ in the assumptions of the abstract Theorem \ref{MainA}, one can not obtain in general the full rigidity statement, as one can see from the proofs given in Section \ref{Sect_AbstrStrThm} below. 
Furthermore, a further application of the standard result in \cite{FCS}, permits to use the same circle of ideas to reach the same conclusions in the more general setting of $f$-stable translators. More precisely we prove the following

\begin{theorem}\label{ThmPoincareTopologyTranslators}
Let $x:\Sigma^{m}\to\mathbb{R}^{m+1}$, $m\geq 2$ be a complete translator satisfying the curvature condition 
\[
\|A\|^2\leq\frac{1+H^2}{4}.
\]
Then $\Sigma$ admits no codimension one cycle that does not disconnects it and it must have only one end. In particular, when $m=2$, $\Sigma$ must be simply connected. 

\noindent More generally, the same conclusions hold by replacing the curvature condition with the request that the immersion is $f$-stable.
\end{theorem}

Reasoning in a similar way and exploiting Theorem \ref{ThmExpTools}, we are able to obtain an analogous structure theorem for self-expanders.

\begin{theorem}\label{ThmPoincareTopologyExpanders}
Let $x:\Sigma^{m}\to\mathbb{R}^{m+1}$, $m\geq 2$ be a complete properly immersed self-expander. If the scalar curvature of $\Sigma$ satisfies 
\[
\ \mathrm{Scal}\geq -(m-1),
\]
then $\Sigma$ admits no codimension one cycle that does not disconnect it and it must have only one end. In particular, when $m=2$, $\Sigma$ must be simply connected.
\end{theorem}

\subsection{Structure of the paper} The paper is organized as follows. Section \ref{Sect_AbstrStrThm} is devoted to the proof of Theorem \ref{MainA}. In Section \ref{SecPoincareTranslators} we focus on translators and we prove both Theorem \ref{Thmfpoincaretranslato} and Theorem \ref{ThmPoincareTopologyTranslators}. In Section \ref{SecPoincareExpanders} we instead concentrate on self-expanders giving a proof of Theorem \ref{ThmExpTools} and Theorem \ref{ThmPoincareTopologyExpanders}. A final appendix contains a detailed discussion of how the results in this paper permit to give a shortcut in order to obtain the full statement of Theorem C and Theorem D in our previous paper \cite{ImperaRimoldi_Transl} (see also the discussion in Appendix A of \cite{IR_MathZ}) and also to refine results in \cite{IR_MathZ}.
\section{Proof of the abstract structure theorem}\label{Sect_AbstrStrThm}

The proof of Theorem \ref{MainA} relies on the following two lemmas. The first is a vanishing result proved in \cite{R1}, which adapts to the weighted setting a result originally obtained in \cite{PRS-JFA05}, \cite{PRS-Book}; see also \cite{PV1}.

\begin{lemma}[Theorem 8 in \cite{R1}]\label{Liouville}
Assume that on a weighted manifold $M_f$ the locally Lipschitz functions $u\geq 0$, $v>0$ satisfy
\begin{equation}\label{ineq_u}
\Delta_fu+a(x)u\geq0
\end{equation}
and
\begin{equation}\label{ineq_deltav}
\Delta_fv+\delta a(x)v\leq0,
\end{equation}
for some constant $\delta\geq1$ and $a(x)\in C^0(M)$. If $u\in L^{2\beta}\left(M_f\right)$, $1\leq\beta\leq \delta$, then there exists a constant $C\geq0$ such that 
\[
\ u^{\delta}=Cv.
\]
Furthermore,
\begin{enumerate}
	\item [(i)]If $\delta>1$ then $u$ is constant on $M$ and either $a\equiv 0$ or $u\equiv 0$;
	\item [(ii)]If $\delta=1$ and $u\not\equiv0$, $v$ and therefore $u^\delta$ satisfy \eqref{ineq_deltav} with equality sign.
\end{enumerate}
\end{lemma}

Recall that, given a Riemannian manifold $M$, we can define the cohomology with compact support  $H_c^k(M)$ as follows. Consider the exact sequence:
\[
\cdots \longrightarrow C^{\infty}_c(\Lambda^{k-1}T^*M)\longrightarrow C^{\infty}_c(\Lambda^{k}T^*M)\longrightarrow C^{\infty}_c(\Lambda^{k+1}T^*M)\longrightarrow\cdots
\] 
Then one can define
\[
H^k_c(M)\doteq\mathrm{ker}\{d C^{\infty}_c(\Lambda^{k}T^*M)\rightarrow C^{\infty}_c(\Lambda^{k+1}T^*M)/dC^{\infty}_c(\Lambda^{k-1}T^*M)\}.
\]

Let $Z$ be a $p$-dimensional oriented, connected, compact manifold without boundary and $\varphi:Z\rightarrow M^n$ be an embedding. The pair $(Z,\varphi)$ is called a $p$-cycle on $M$. We say that an $(m-1)$-cycle (aka codimension one cycle) $Z$ does not disconnect $M$ if $M\backslash \varphi(Z)$ is connected. 
\medskip

It is a well known fact that the first cohomology group with compact support gives a control on the topology of the manifold $M$. This is the content of the following proposition, which combines Lemma 2.1 in \cite{Carron_LectNotes} and Proposition 3.3 in \cite{KS}.
\begin{proposition}[\cite{Carron_LectNotes}, \cite{KS}]\label{propH1candtopology}
Let $(M, g)$ be a non-compact connected Riemannian manifold. Then
\begin{itemize}
\item[(i)] If $M$ has at least $k$ ends then
\[
\ \mathrm{dim}H_{c}^{1}(M)\geq k-1;
\]
\item[(ii)] If there exists a codimension one cycle which does not disconnect $M$, then there exists a closed $1$-form $\alpha$ with compact support and a $1$-cycle $\gamma$ such that
\[
 \int_{\gamma}\alpha=1. 
\]
Therefore $H_{c}^{1}(M)\neq \left\{0\right\}$.
\end{itemize}
\end{proposition}

Let us now denote by $\mathcal{H}^{1}_{f}(M)$ the space of $f$-harmonic $1$-forms which are square integrable with respect to the weighted measure:
\[
\ \mathcal{H}^{1}_{f}(M)\doteq\left\{\omega\in\Lambda^{1}T^{*}M\,:\,d\omega=\delta_{f}\omega=0,\,\,\int_{M}|\omega|^{2}e^{-f}d\mathrm{vol}<+\infty\right\},
\]
where $\delta_{f}=\delta+i_{\nabla f}$ and $\delta$ is the usual codifferential. As a consequence of the previous proposition we can prove the following

\begin{lemma}\label{nontrivialfharmonic1forms}
Let $M_f$ be a complete weighted manifold. Assume that $M_f$ enjoys the weighted $f$-Poincar\'e inequality \eqref{eq_weightedpoinc} for some positive smooth function $\rho$ satisfying, on any end $E$ of $M$,
\begin{equation*}
\int_{E}\rho e^{-f}d\mathrm{vol}=+\infty.
\end{equation*}
Suppose that one of the following assumptions is satisfied:
\begin{itemize}
    \item[(i)] $M$ has at least two ends;
   \item[(ii)] $M$ admits a codimension one cycle that does not disconnect it.
\end{itemize}
Then $\mathcal{H}^{1}_{f}(M)\neq \left\{0\right\}$.
\end{lemma}

\begin{proof}
We claim that if 
\[
\ \int_{M}\rho\varphi^{2}e^{-f}d\mathrm{vol}\leq \int_{M}|\nabla \varphi|^{2}e^{-f}d\mathrm{vol},
\] 
for any $\varphi\in C_{c}^{\infty}(M)$, and \eqref{Nonintrho} holds, then
\[
\ \mathrm{dim}\mathcal{H}^{1}_{f}(M)\geq \mathrm{dim}H^{1}_{c}(M). 
\]
Once proved this claim, the result is a direct consequence of Proposition \ref{propH1candtopology}.

As for the claim, note that in \cite{Bue} it was proved the following decomposition of the space $L^{2,f}(\Lambda^1T^*M)$ of square integrable one forms on $M$ with respect to the weighted measure:
\[
L^{2,f}(\Lambda^1T^*M)=A\oplus B_f \oplus \mathcal{H}^{1}_{f}(M),
\]
where
\[
\begin{cases}
A=\overline{\{dg: g\in C^\infty_c(M)\}},\\
B_f=\overline{\{\delta_f\eta: \eta\in C^{\infty}_c(\Lambda^2T^*M)\}},
\end{cases}
\]
the closure been taken with respect to the weighted $L^2$-norm.
Denote by $Z^1_f(M)$ the space of $L^{2,f}$ closed $1$-forms, i.e. 
\[
Z^1_f(M)=\{\omega \in L^{2,f}(\Lambda^1T^*M): d\omega=0\}.
\]
It then follows that 
\[
H^1_f(M)\doteq\frac{Z^1_f(M)}{A}\simeq \mathcal{H}^{1}_{f}(M).
\]
We now consider the map
\[
i:H^1_c(M)\rightarrow H^1_f(M).
\]
and we claim that $i$ is injective. This is equivalent to prove that if $\beta$ is a closed $1$-form which is zero with respect to the $L^{2,f}$ cohomology, then there exists $v\in C_{c}^{\infty}(M)$ such that $\beta=dv$. First note that if $\beta$ is zero with respect to the $L^{2,f}$ cohomology then there exists a sequence $v_j\in C^{\infty}_c(M)$ such that
\[
\lim_{j\rightarrow+\infty}\|\beta-d v_j\|=0. 
\]
Lemma 3.1 in \cite{KS} implies that there exists $v\in C^{\infty}(M)$ such that $dv=\beta$. Moreover, the validity of the weighted $f$-Poincar\'e inequality \eqref{eq_weightedpoinc} implies that $\{v_j\}$ converges to the function $v$ in $L^2(M, \rho e^{-f}d\mathrm{vol})$. Since $\beta$ has compact support, $v$ must be locally constant at infinity. On the other hand, since $v\in L^2(M, \rho e^{-f}d\mathrm{vol})$ and \eqref{Nonintrho} holds, this implies that $v$ must have compact support, proving the claim.

\noindent To conclude the proof note that the injectivity of the map $i$ implies that
\[
\mathrm{dim}(H_c^1(M))\leq \mathrm{dim}(\mathcal{H}^{1}_{f}(M)).
\]

\end{proof}

\begin{proof}[Proof of Theorem \ref{MainA}]
We have that the validity of the weighted $f$-Poincar\'e inequality \eqref{eq_weightedpoinc} implies the existence of a function $v>0$ satisfying the equality
\[
\Delta_f v+\rho v=0.
\]
As a consequence note that $M$ must have infinite $f$-volume. Indeed if we assume by contradiction that $\mathrm{vol}_{f}(M)<\infty$ then $M_f$ is $f$-parabolic, i.e. any $f$-superharmonic function which is bounded from below must be constant; see e.g. \cite{Grig_Weight}. This would imply that the function $v$ above should be constant and thus that $\rho$ should vanish identically, which is not the case. Moreover, since $\rho(x)\geq \delta a(x)$, we obtain that the function $v$ is a solution of inequality \eqref{ineq_deltav}. 

Assume now by contradiction that either $\Sigma$ has at least two ends or that there exists a codimension one cycle that does not disconnect $M$. Lemma \ref{nontrivialfharmonic1forms} gives that there exists a non-trivial $\omega\in \mathcal{H}^1_f(M)$. As a consequence of $f$-Weitzenb\"ock and Kato's inequalities, it holds that
\begin{align*}
|\omega|\Delta_f |\omega|&=\Delta_f\frac{|\omega|^2}{2}-|\nabla |\omega||^2\\
&=|\nabla \omega|^2-|\nabla |\omega||^2+\mathrm{Ric}_f(\omega^\sharp,\omega^\sharp)\\
&\geq-a(x)|\omega|^2.
\end{align*}
Hence the function $u=|\omega|$ is a solution of inequality \eqref{ineq_u} satisfying $u\in L^{2\beta}(M_f)$ with $\beta=1$. We can then apply Lemma \ref{Liouville} and conclude that $u=|\omega|$ is a non-zero constant. However, this is in contradiction to the fact that $M$ has infinite $f$-volume.
\end{proof}

\section{Poincar\'e inequality and topological rigidity of translators under stability assumptions}\label{SecPoincareTranslators}

Let $x:\Sigma^m\to\mathbb{R}^{m+1}$ be a translator for the mean curvature flow. We first recall some basic identities that will be useful in what follows. For a proof we refer to \cite[Lemma 2.1]{IR_MathZ}.

\begin{lemma}\label{LemBasEq}
Set $f=-\<x,\bar{E}_{m+1}\>$ and denote by $\nu$ the unit normal vector to $\Sigma$, by $A$ the second fundamental form of the immersion, and by $E_{i}$ the projections of $\bar{E}_{i}$ on $T\Sigma$. Let $\xi\in \chi(\Sigma)$, then 
\begin{align*}
\nabla_{\xi}E_{i}=&\<\bar{E}_{i}, \nu\>A\xi;\\
\nabla\<\bar{E}_{i}, \nu\>=&-AE_{i};\\
\Delta_{f}\<\bar{E_{i}},\nu\>=&-\<\bar{E}_{i}, \nu\>\|A\|^2;\\
\Delta_{f}\<x,\bar{E}_{i}\>=&\<\bar{E}_{i},\bar{E}_{m+1}\>.
\end{align*}
\end{lemma}

The next lemma is an adaptation to the weighted setting of Proposition 1.1 in \cite{LiWang_Poincare}.
\begin{lemma}\label{lemmaweightedpoincare}
Let $M_f$ be a weighted manifold. Assume that there exists a nonnegative function $g$ defined on $M$, that is not identically $0$, satisfying
\[
\Delta_f g\leq -\rho g,
\]
with $\rho$ a non-negative function on $M$. Then the following weighted $f$-Poincar\'e inequality holds
\[
\int_M \rho\varphi^2 e^{-f}d\mathrm{vol}\leq \int_M |\nabla \varphi|^2 e^{-f}d\mathrm{vol}
\]
for all $\varphi\in C^\infty_c(M)$.
\end{lemma}
\begin{proof}
Let $\Omega\subset M$ be a smooth compact subdomain of $M$. Denote by $\lambda_1^f(\rho,\Omega)$ the first Dirichlet eigenvalue on $\Omega$ for the operator
\[
\Delta_f+\rho,
\] 
and let $u$ be the first eigenfunction satisfying
\[
\begin{cases}
\Delta_f u+\rho u=-\lambda_1^f(\rho,\Omega) u & \textrm{on } \Omega\\
u\equiv 0 & \textrm{on } \partial \Omega.
\end{cases}
\]
In particular, $u>0$ in $\mathrm{int}(\Omega)$ and, as a consequence of the Hopf Lemma, $\partial u/\partial \nu\leq 0$ on $\partial \Omega$. Here $\nu$ denotes the outward unit normal of $\partial \Omega$. Integration by parts yields:
\begin{align*}
\lambda_1^f(\rho,\Omega) \int_\Omega ug e^{-f}d\mathrm{vol}& \geq \int_\Omega(u\Delta_f g-g\Delta_f u) e^{-f}d\mathrm{vol}\\
&= \int_{\partial \Omega} u\frac{\partial g}{\partial \nu} e^{-f}d\mathrm{vol}_{m-1}-
 \int_{\partial \Omega} g\frac{\partial u}{\partial \nu} e^{-f}d\mathrm{vol}_{m-1}\\
 &\geq 0.
 \end{align*}
 Since both $u>0$ and $g\geq 0$ are not identically $0$, this implies that $\lambda_1^f(\rho,\Omega)\geq 0$. Now, by the min-max principle,
 \begin{align*}
 0 &\leq \lambda_1^f(\rho,\Omega)\int_\Omega \varphi^2 e^{-f}d\mathrm{vol}\\
 &\leq \int_\Omega \left(|\nabla \varphi|^2-\rho \varphi^2\right) e^{-f}d\mathrm{vol},
 \end{align*}
 for all $\varphi\in C^{\infty}_0(\Omega)$. Since $\Omega $ is arbitrary, this implies the validity of the weighted $f$-Poincar\'e inequality on the whole $M$.
\end{proof}

As a consequence of Lemma \ref{lemmaweightedpoincare} we get the following 
\begin{theorem}
Let $x:\Sigma^{m\geq 2}\to\mathbb{R}^{m+1}$ be a translator for the mean curvature flow. Then the following $f$-Poincar\'e inequality holds
\begin{equation}\label{eqfpoincaretranslators}
\frac{1}{4}\int_\Sigma \varphi^2 e^{-f}d\mathrm{vol}_{\Sigma}\leq \frac{1}{4}\int_\Sigma (1+H^2) \varphi^2 e^{-f}d\mathrm{vol}_{\Sigma}\leq \int_\Sigma |\nabla \varphi|^2 e^{-f}d\mathrm{vol}_{\Sigma}
\end{equation}
for any $\varphi\in C^{\infty}_c(\Sigma)$.
\end{theorem}
\begin{proof}
Set $h=\langle x,E_{m+1}\rangle$ and let $\alpha:\mathbb{R}\to\mathbb{R}$, $\alpha(t)=e^{-t/2}$. By Lemma \ref{LemBasEq} we have that $\Delta_f h=1$ and
\begin{align*}
\Delta_f (\alpha\circ h)&=\alpha'(h)\Delta_f h+\alpha''(h)|\nabla h|^2\\
&= -\frac{1}{2}e^{-h/2}+\frac{1}{4}e^{-h/2}(1-H^2)\\
&=-\frac{1}{4}e^{-h/2}-\frac{1}{4}e^{-h/2}H^2\\
&= -\frac{1}{4}(1+H^2)(\alpha\circ h).
\end{align*}
The conclusion follows now by Lemma \ref{lemmaweightedpoincare} with the choices $g=\alpha\circ h$, $\rho=(1+H^2)/4$.
\end{proof}

On the other hand, a direct argument yields the following

\begin{proposition}\label{InfVolTransl}
Let $x:\Sigma^{m\geq 2}\to\mathbb{R}^{m+1}$ be a complete translator. Then every end of $\Sigma$ has infinite $f$-volume. 
\end{proposition}
\begin{proof}
Let $o$ be a reference point on $\Sigma$ and let $E\subseteq \Sigma\backslash B^\Sigma_{R_0}(o)$ be an end with respect to the compact set $B^\Sigma_{R_0}(o)$. Given $R>R_{0}$, denote by $E_R=E\cap B^\Sigma_{R}(o)$. Given $\varepsilon>0$, let $\varphi$ be a smooth cut-off function satisfying
\[
\varphi=\begin{cases}
1 & \textrm{on } E_R\\
0 & \textrm{on } E \setminus E_{R+\varepsilon}
\end{cases},
\quad |\nabla\varphi|\leq \frac{C}{\varepsilon},
\]
where $C$ is a constant that does not depend on $R$.
It follows by Lemma \ref{LemBasEq} that
\begin{align*}
\mathrm{div}(\varphi^2e^{\frac{x_3}{2}}\nabla e^{\frac{x_3}{2}})&=\varphi^2 e^{\frac{x_3}{2}}\Delta e^{\frac{x_3}{2}}+\varphi^2 |\nabla e^{\frac{x_3}{2}}|^2+e^{x_3}\varphi\langle \nabla \varphi, E_{3}\rangle\\
&=\varphi^2\frac{1+H^2}{4}e^{x_3}+\varphi^2\frac{1-H^2}{4}e^{x_3}+e^{x_3}\varphi\langle \nabla \varphi, E_{3}\rangle\\
&=\frac{\varphi^2}{2}e^{x_3}+e^{x_3}\varphi\langle \nabla \varphi, E_{3}\rangle.
\end{align*}
Integrating the previous identity and denoting by $\nu_{E}$ the outward pointing unit normal to $\partial E$, we get
\begin{align*}
\int_{\partial E} \frac{1}{2}\langle E_3,\nu_E\rangle e^{x_3}d\mathrm{vol}_{m-1}&=\int_{E_{R+\varepsilon}}\left(\frac{\varphi^2}{2}+\varphi\langle \nabla \varphi, E_{3}\rangle\right)e^{x_3}d\mathrm{vol}\\
&\geq \frac12 \mathrm{vol}_f(E_R)-\frac{C}{\varepsilon}\mathrm{vol}_f(E_{R+\varepsilon}\backslash E_R)
\end{align*}
Letting $\varepsilon\rightarrow 0$ and using the coarea formula we thus obtain
\[
\frac{d\mathrm{vol}_f(E_R)}{dR}-\frac{1}{2C}\mathrm{vol}_f(E_R)+\frac{1}{2C} \mathrm{vol}_f(\partial E)\geq0.
\]
Let $w(R)$ be the unique solution of the Cauchy problem
\[
\begin{cases}
\frac{dw(R)}{dR}-\frac{1}{2C}w(R)+\frac{1}{2C} \mathrm{vol}_f(\partial E)=0,\\
w(R_0)=\mathrm{vol}_f(E_{R_0}).
\end{cases}
\]
 Then
\[
\mathrm{vol}_f(E_R)\geq w(R)=A(R_0,\mathrm{vol}_f(\partial{E}),C)e^{\frac{R}{2C}}+\mathrm{vol}_f(\partial E),
\]
for some constant $A$ depending on $R_{0}, \mathrm{vol}_{f}(\partial E)$ and $C$. Letting $R\rightarrow \infty$ in the previous inequality we obtain that the end $E$ must have infinite $f$-volume. 
\end{proof}

\begin{remark}\label{RmkEndsNonfPar}
\rm{Recall that an end $E$ of a weighted manifold $M_f$ with respect to a relatively compact domain $\Omega$ is said to be non-$f$-parabolic if there exists a non-constant function $u$ satisfying
\[
\begin{cases}
\Delta_f u\geq 0 & \textrm{on }E\\
\sup_E u<+\infty\\
\frac{\partial u}{\partial \nu_E}\leq 0.
\end{cases}
\]
The end is said to be $f$-parabolic otherwise. By \cite[Lemma 1]{IR_fMin}\label{f-CSZ} (with $\alpha=0$) we know that every end of a weighted manifold satisfying an $f$-Poincar\'e inequality is either non-$f$-parabolic or it has finite $f$-volume. Thanks to Proposition \ref{InfVolTransl}, we can hence conclude that every end of a complete translator  $x:\Sigma^{m\geq 2}\to\mathbb{R}^{m+1}$ is non-$f$-parabolic.}
\end{remark}

Using Theorem \ref{Thmfpoincaretranslato} and Proposition \ref{InfVolTransl} we are able to prove Theorem \ref{ThmPoincareTopologyTranslators}.

\begin{proof}[Proof of Theorem \ref{ThmPoincareTopologyTranslators}]
We argue by contradiction and assume that either $\Sigma$ has at least two ends or it admits a codimension one cycle that does not disconnect it. We reason as in the proof of Theorem \ref{MainA}, with the choices $\rho=\frac{1+H^2}{4}$, $v>0$ as in Theorem \ref{MainA}, $a=-\|A\|^2$, and $u=|\omega|\in\mathcal{H}^1_f(\Sigma)$. Letting $\delta=1$ we can then apply part $(ii)$ of Lemma \ref{Liouville}, to obtain that $v=|\omega|$ and that the inequalities \eqref{ineq_deltav} and \eqref{ineq_u} must hold with equality sign. However, this implies that
\[
-|A\omega^\sharp|^2=\mathrm{Ric}^\Sigma_f(\omega^\sharp,\omega^\sharp)=-\|A\|^2|\omega|^2.
\]
Fix an arbitrary point $p\in \Sigma$ and let $\{e_{i}(p)\}_{i=1}^{m}$ be an orthonormal basis for $T_{p}\Sigma$ such that $e_{1}(p)=\frac{\omega^\sharp}{|\omega|}(p)$. Since, at $p$, 
\[
\ \|A\|^2=|Ae_1|^2+\sum_{i=2}^{m}|Ae_{i}|^2,
\]
we get that
\[
\ \sum_{i=2}^{m}|Ae_{i}|^2=0,
\]
and this implies that $\|A\|^2(p)=H^2(p)$. Since $p$ is arbitrary this yields that $\|A\|^2=H^2$ on $\Sigma$. In particular, as a consequence of Gauss' equation, the scalar curvature of $\Sigma$ vanishes identically. Using Corollary 2.4 in \cite{MSHS} we conclude that $\Sigma$ has to be either a member of the tilted grim reaper family or a translator hyperplane. However all these hypersurfaces have only one end and do not admit any codimension one cycle that does not disconnect them.

As for the last part of the statement, let us assume that the translator is $f$-stable, that is, the operator $L_{f}$ defined above is nonnegative. Then by \cite{FCS} (see also \cite{IR_fMin}) there exists $v>0$ satisfying \eqref{ineq_deltav} with $\rho=a=-\|A\|^2$ and $\delta=1$ and we can reason exactly as above.
\end{proof}

\section{Poincar\'e inequality and topological rigidity of self-expanders under curvature assumptions}\label{SecPoincareExpanders}

In the following lemma we collect some basic equations valid on any self-expander.
\begin{lemma}\label{BasEqSE}
Let $x:\Sigma^{m\geq 2}\to\mathbb{R}^{m+1}$ be a self-expander of the mean curvature flow and let $f=-|x|^2/2$. Then
\begin{enumerate}
\item $\nabla f=-x+H\nu$;
\item $\Delta_f f=-m+2f$;
\item $\mathrm{Ric}^\Sigma_f\geq -1-\|A\|^2$.
\end{enumerate}
\end{lemma} 
\begin{proof}
Let $\Sigma$ be a self-expander for the mean curvature flow. Then, 
\begin{align*}
\nabla f&=\nabla\left(-\frac{|x|^2}{2}\right)\\
&=-x^T=-(x-\langle x,\nu\rangle \nu).
\end{align*}
Point $(1)$ is a direct consequence of the previous computation and of the expander equation \eqref{SE}. As for point $(2)$, let $Y\in T\Sigma$. Then
\begin{align*}
\mathrm{Hess} f(Y,Y)&=\langle \nabla_Y \nabla f,Y\rangle\\
&=\langle \overline{\nabla}_Y \nabla f,Y\rangle\\
&=-\langle \overline{\nabla}_Y x,Y\rangle+H\langle \overline{\nabla}_Y \nu,Y\rangle\\
&=-|Y|^2-H\langle AY,Y\rangle,
\end{align*}
where we are denoting by $\overline{\nabla}$ the connection of the ambient space. Hence
\begin{align*}
\Delta_f f&=\Delta f-|\nabla f|^2\\
&=-m-H^2-(|x|^2-H^2)\\
&=-m+2f.
\end{align*}
Finally, as for point $(3)$, note that by Gauss' equation
\begin{align*}
    \mathrm{Ric}_f^\Sigma(Y,Y)&=\mathrm{Ric}^\Sigma(Y,Y)+\mathrm{Hess} f(Y,Y)\\
    &=H\langle AY,Y\rangle-|AY|^2-|Y|^2-H\langle AY,Y\rangle\\
    &\geq -(1+\|A\|^2)|Y|^2.    
\end{align*}
\end{proof}

The following two results hold.

\begin{proposition}[Theorem 1.2 in \cite{ChengZhou_Selfexp}]\label{propfpoincareexpanders}
Let $x:\Sigma^{m\geq 2}\to\mathbb{R}^{m+1}$ be a self-expander of the mean curvature flow and let $f=-|x|^2/2$. Then the following $f$-Poincar\'e inequality holds
\begin{equation}\label{eqfpoincareexpanders}
m\int_\Sigma \varphi^2 e^{-f}d\mathrm{vol}_{\Sigma}\leq \int_\Sigma (m+H^2) \varphi^2 e^{-f}d\mathrm{vol}_{\Sigma}\leq \int_\Sigma |\nabla \varphi|^2 e^{-f}d\mathrm{vol}_{\Sigma},
\end{equation}
for any $\varphi\in C^{\infty}_c(\Sigma)$.
\end{proposition}
\begin{proof}
By Lemma \ref{BasEqSE} we have that 
\begin{align*}
\Delta_f e^{f}=&e^f\left(\Delta_f f+|\nabla f|^2\right)\\
=& -e^f \left(m+H^2\right).
\end{align*}
The conclusion follows now by Lemma \ref{lemmaweightedpoincare} with the choices $g=e^f$, $\rho=m+H^2$.
\end{proof}

\begin{proposition}
Let $x:\Sigma^{m\geq 2}\to\mathbb{R}^{m+1}$ be a properly immersed self-expander of the mean curvature flow and let $f=-|x|^2/2$. Then any end of $\Sigma$ has infinite $f$-volume.
\end{proposition}
\begin{proof}
Let $E$ be an end of $\Sigma$ with respect to a compact set $K$. Consider the function $e^{-\frac{|x|^2}{4}}=e^{\frac{f}{2}}$. A simple computation using Lemma \ref{BasEqSE} and \eqref{SE} shows that
\begin{align*}
\Delta_f e^{\frac{f}{2}}=&\frac{\Delta_{f}e^{f}}{2e^{\frac{f}{2}}}-\frac{|\nabla e^{f}|^2}{4e^{\frac{3f}{2}}}\\
&= e^{\frac{f}{2}} \left(-\frac{m}{2}-\frac{H^2}{2}-\frac{|x|^2}{4}+\frac{H^2}{4}\right)\\
&\leq 0.
\end{align*}
Hence $u=e^{\frac{f}{2}}$ is a positive $f$-superharmonic function. Assume by contradiction that $E$ has finite $f$-volume. Then, as a consequence of Remark \ref{RmkEndsNonfPar}, E must be $f$-parabolic and hence any $f$-superharmonic function bounded from below must attain its infimum on the boundary of $E$. On the other hand, since $\Sigma$ is properly immersed, $\sup_E |x|^2=+\infty$ and hence $\inf_E  u=0<\inf_{\partial E} u$, leading to a contradiction.
\end{proof}

With these tools, we can show, as promised, topological rigidity for self-expanders under curvature assumptions.

\begin{proof}[Proof of Theorem \ref{ThmPoincareTopologyExpanders}]
 We argue by contradiction and assume that $\Sigma$ has at least two ends and it admits a codimension one cycle that does not disconnect it. We use Gauss' equation and reason as in the proof of Theorem \ref{MainA}, with the choices $\rho=m+H^2$, $v>0$ as in Theorem \ref{MainA}, $a=-\|A\|^2-1$, $u=|\omega|\in\mathcal{H}^1_f(\Sigma)$. In this case we can choose $\delta=1$ and apply part $(ii)$ of Lemma \ref{Liouville}, to obtain that $v=|\omega|$ and that the inequalities \eqref{ineq_deltav} and \eqref{ineq_u} must hold with equality sign. However, this implies that
\[
-|\omega|^2-|A\omega^\sharp|^2=\mathrm{Ric}^\Sigma_f(\omega^\sharp,\omega^\sharp)=-|\omega|^2-\|A\|^2|\omega|^2.
\]
Fix an arbitrary point $p\in \Sigma$ and let $\{e_{i}(p)\}_{i=1}^{m}$ be an orthonormal basis for $T_{p}\Sigma$ such that $e_{1}(p)=\frac{\omega^\sharp}{|\omega|}(p)$. Hence at $p$, 
\[
\ \|A\|^2=|A(e_1)|^2+\sum_{i=2}^{m}|Ae_{i}|^2,
\]
from which we obtain
\[
\ \sum_{i=2}^{m}|Ae_{i}|^2=0.
\]
We thus deduce that $\|A\|^2(p)=H^2(p)$. Since $p$ is arbitrary this implies that $\|A\|^2=H^2$ on $\Sigma$. In particular, as a consequence of Gauss' equation, the scalar curvature of $\Sigma$ vanishes identically. Using Theorem 6.1 in \cite{AncariCheng_GeoDed} we conclude that $\Sigma=\Gamma\times\mathbb{R}^{m-1}$, where $\Gamma$ is a complete self-expander curve immersed in $\mathbb{R}^2$. However, this contradicts the assumption that either $\Sigma$ has at least two ends or that it admits a codimension one cycle that does not disconnect it.   
\end{proof}

\begin{remark}
\rm{In \cite[Corollary 1.1]{AncariCheng_GeoDed} it is proved that the condition on the scalar curvature $\mathrm{Scal}\geq -(m-1)$ implies $f$-stability of self-expanders. One could be tempted, also in this setting, to use directly the $f$-stability in order to obtain topological rigidity. However in this case
\[
\ \mathrm{Ric}_{f}\geq -1-\|A\|^2,
\]
while the $f$-stability assumption could be phrased, through \cite{FCS}, by the existence of $v>0$ such that
\[
\Delta_{f}v+(\|A\|^2-1)v=0.
\]
Hence one immediately realises that it is not possible to apply the vanishing theorem to conclude.}
\end{remark}


\appendix
\section{A note about theorems C and D in \cite{ImperaRimoldi_Transl}}

This appendix is devoted to a comparison of the approach taken in the present paper with that of our previous papers \cite{ImperaRimoldi_Transl} and \cite{IR_MathZ}, where we established estimates on the $f$-index of translators via topology. Moreover, we explicitly note that the approach via Poincar\'e inequality permits to fix a gap in the proof of Theorem C and Theorem D in \cite{ImperaRimoldi_Transl} (see Appendix A in \cite{IR_MathZ}) and to get a proof of the full statement of these results without any additional assumption.

Recall that the $f$-index of a translator $x:\Sigma^{m}\to\mathbb{R}^{m+1}$ is defined in terms of the generalized Morse index of $L_{f}$ on $\Sigma$. Namely, given a relatively compact domain $\Omega\Subset\Sigma$ we define
\[
\ \mathrm{Ind}^{L_{f}}(\Omega)=\sharp\left\{\mathrm{negative\,eigenvalues\,of\,}L_{f}\,\mathrm{on\,} C_{0}^{\infty}(\Omega)\right\}.
\]
The $f$-index of $\Sigma$ is then defined as
\[
\ \mathrm{Ind}_{f}(\Sigma)\doteq\mathrm{Ind}^{L_{f}}(\Sigma)=\sup_{\Omega\Subset\Sigma}\mathrm{Ind}^{L_{f}}(\Omega).
\]

If $\mathrm{Ind}_{f}(\Sigma)<\infty$ we note that it is a consequence of \cite[Theorem 3]{IR_fMin} (see also the considerations before Corollary 1 in \cite{IR_fMin}) that $\mathrm{dim}\left(\mathcal{H}^{1}_{f}(\Sigma)\right)<\infty$.

In \cite{ImperaRimoldi_Transl} and \cite{IR_MathZ}, we used an adaptation to the weighted setting of Li-Tam's theory, \cite{LiTam}, that allows to control the number of non-$f$-parabolic ends in terms of the dimension of the space of bounded $f$-harmonic functions with finite Dirichlet energy. This latter in turn is controlled from above by $\mathrm{dim}(\mathcal{H}^{1}_{f}(\Sigma))$. Hence, in order to obtain a control of the $f$-index via the topology at infinity we used the validity of an $L^{2,f}$ Sobolev inequality (which we are able to guarantee assuming an upper halfspace confinement for $x(\Sigma)$ and the dimensional restriction $m\geq 3$)  in order to prove that all the ends are in fact non-$f$-parabolic. 
\medskip

In the present paper, the validity of the $f$-Poincar\'e inequality and the fact that all the ends have infinite $f$-volume permit us to bypass Li-Tam's theory, thus working directly with the first cohomology group with compact support. This in particular permits to avoid the upper half-space assumption and the restriction on the dimension in both the qualitative index estimate (Theorems C and D in \cite{ImperaRimoldi_Transl}; see also the remarks in Appendix A in \cite{IR_MathZ}) and the quantitative estimate for the generalized index (Theorem A in \cite{IR_MathZ}). In particular, as for the unstable case, we are able to prove the following

\begin{proposition}
Let $x:\Sigma^{m\geq 2}\to\mathbb{R}^{m+1}$ be a complete translator. If $\Sigma$ has finite $f$-index then it has finitely many ends. Moroever if $\|A\|^2\in L^{\infty}(\Sigma)$ then
\[
\ \mathrm{Ind}_{f}(\Sigma)+\mathrm{Null}_{f}(\Sigma)\geq \frac{2}{m(m+1)}\mathrm{dim}(\mathcal{H}^{1}_{f}(\Sigma))\geq \frac{2}{m(m+1)}(\sharp\left\{ends\right\}-1,)
\]
where $\mathrm{Null}_{f}(\Sigma)$ is the dimension of the space of Jacobi functions which are square integrable with respect to the weighted measure.
\end{proposition}

\begin{remark}
\rm{Since, by Remark \ref{RmkEndsNonfPar} all the ends of a translator are actually non-$f$-parabolic, one could also proceed as in \cite{ImperaRimoldi_Transl}, and through an adaptation to the weighted setting of Li-Tam's theory get to the same conclusion in a less direct way.}
\end{remark}

\bigskip

\bibliographystyle{amsplain}
\bibliography{WeightedPoincare}
 
\end{document}